\documentclass[reqno,11pt,oneside]{amsart}
\usepackage{geometry}
\usepackage{amsmath, amsthm, amssymb, mathtools, mathrsfs, stmaryrd}
\usepackage{amsfonts}
\usepackage{ascmac}
\usepackage{mathrsfs} 
\usepackage{xcolor}
\definecolor{HY}{RGB}{62,78,200}
\definecolor{HYH}{RGB}{105,90,190}
\usepackage[dvipdfmx, colorlinks=true, linkcolor=HY, citecolor=HYH, urlcolor=HY]{hyperref}
\usepackage{latexsym}
\usepackage{amssymb}
\address[Anju Yokoi]{Ikeda Senior High School Attached to Osaka Kyoiku University, 1-5-1, Midorigaoka, Ikeda-shi, Osaka, 563-0026, Japan}
\email{anju.scorpion@icloud.com}
\theoremstyle{definition}
\newtheorem{dfn}{Definition}[section]

\newtheorem{lem}[dfn]{Lemma}
\newtheorem{thm}[dfn]{Theorem}

\newtheorem{rem}[dfn]{Remark}

\makeatletter

\@addtoreset{equation}{section}
\makeatother
\allowdisplaybreaks
\title{An interpolation of Bradley's sum formula}
\author{Anju Yokoi}
\begin{document}
\maketitle
\begin{abstract}
  The sum formula for $q$-multiple zeta values is a well-known relation. In this paper, we present its generalization for the $q$-multiple zeta function.
\end{abstract}
\phantom{a}
\section{Introduction}
Let $0<q<1$. For any positive integer $m$, the $q$-analog of $m$ is defined by $[m]_q=(1-q^m)/(1-q)$. The $q$-analogue of Euler--Zagier multiple zeta functions (hereafter $q$MZFs) is defined by 
\[\zeta_q(s_1,\ldots,s_r)\coloneqq\sum_{0<n_1<\cdots<n_r}\prod_{j=1}^r\frac{q^{(s_j-1)n_j}}{[n_j]_q^{s_j}},\]
where $s_j\in\mathbb{C}$ for $j=1,\ldots,r$. The number of the variables $r$ is called the depth of $\zeta_q(s_1,\ldots,s_r)$.
 Zhao \cite{Zhao} proved that the series is absolutely convergent in the domain
\[\{(s_1,\ldots,s_r)\in\mathbb{C}^r\mid\Re(s_{r-k+1}+\cdots+s_r)>k\quad(k=1,\ldots,r)\}\]
and can be analytically continued to $\mathbb{C}^r$ as a meromorphic function with simple poles given by
\begin{equation*}
  \mathfrak{S}_r\coloneqq \left\{ (s_1,\ldots,s_r)\in\mathbb{C}^r \middle|
  \begin{aligned} 
    s_r\in1+\frac{2\pi i}{\log q}\mathbb{Z},\text{or}\; s_r\in\mathbb{Z}_{\le0}+\frac{2\pi i}{\log q}\mathbb{Z}_{\not=0}\\
    \;\text{or}\; s_j+\cdots+s_r\in\mathbb{Z}_{\le r-j+1}+\frac{2\pi i}{\log q}\mathbb{Z},\; j<r
  \end{aligned}
  \right\}
\end{equation*} for $1\le j\le r$. Here the last part in $\mathfrak{S}_r$ is vacuous if $r=1$.
According to \cite[Main Theorem]{Zhao}, taking the limit $q\uparrow1$ of the $q$MZFs yields the Euler--Zagier multiple zeta functions (hereafter MZFs)
\[\lim_{q\uparrow1}\zeta_q(s_1,\ldots,s_r)=\zeta(s_1,\ldots,s_r)\coloneqq\sum_{0<n_1<\cdots<n_r}\prod_{j=1}^r\frac{1}{n_j^{s_j}}.\]
The special values $\zeta_q(k_1,\ldots,k_r)$ with $k_j\in\mathbb{Z}_{>0}$ for $j=1,\ldots,r-1$ and $k_r\ge2$ are called $q$-multiple zeta values (hereafter, $q$MZVs). We call such tuple $(k_1,\ldots,k_r)$ the admissible index. Among the many known $\mathbb{Q}$-linear relations for $\zeta_{q}(k_1, \ldots, k_r)$, the $q$-sum formula is regarded as one of the most fundamental. 

\begin{thm}[{Sum formula for $q$MZVs; Bradley \cite{Dav}}]\label{Theorem 1}
For integers $k$ and $r$ with $1\le r\le k-1$, set
\[\mathcal{I}_0(k,r)\coloneqq\{(k_1,\ldots,k_r)\in\mathbb{Z}_{>0}^r\mid k_1+\cdots+k_r=k,k_r\ge2\}.\]
Then we have
  \begin{align*}
    \sum_{\mathbf{k}\in \mathcal{I}_0(k,r)}\zeta_q(\mathbf{k})=\zeta_q(k).
  \end{align*}
\end{thm}
Among the fundamental relations for multiple zeta values (hereafter MZVs), the sum formula stands out as one of the most basic. Hirose--Murahara--Onozuka \cite{HMO} obtained the following theorem through complex analytic interpolation of this formula.
\begin{thm}[{Sum formula for MZFs; Hirose--Murahara--Onozuka \cite[Theorem 1.2]{HMO}}]\label{Theorem 2}
For $s\in\mathbb{C}$ with $\Re(s)>1$ and $s\not=2$ we have
\[\sum_{n=0}^\infty(\zeta(s-n-2,n+2)-\zeta(-n,s+n))=\zeta(s).\]
\end{thm}
In this paper, we give a generalization of Theorem \ref{Theorem 1} and Theorem \ref{Theorem 2}. Here we set $\mathfrak{A}\coloneqq\{s\in\mathbb{C}\mid s\in\mathbb{Z}_{\le2}+\frac{2\pi i}{\log q}\mathbb{Z}\}$.
\begin{thm}\label{thm1}
For $s\in\mathbb{C}\setminus\mathfrak{A}$, we have
  \begin{align*}
\sum_{n=0}^\infty\left(\zeta_q(s-n-2,n+2)-\zeta_q(-n,s+n)\right)&=\zeta_q(s).
\end{align*}
\end{thm}
\begin{rem}
  This theorem holds in a wider range than Theorem \ref{Theorem 2}.
  By specializing a complex number $s$ at a positive integer $k$ satisfying $k>2$, we obtain the depth-$2$ case of Theorem \ref{Theorem 1}. Furthermore, taking the limit as $q\to1$ yields Theorem \ref{Theorem 2}.
\end{rem}
This also can be generalized to the arbitrary depth.
\begin{dfn}
  For a non-negative integer $a$ and a positive integer $b$, we define $G_q^{(a,b)}(s_1,\ldots,s_a;s)$ inductively by
  \begin{align*}
    G_q^{(a,1)}(s_1,\ldots,s_a;s)&\coloneqq\zeta_q(s_1,\ldots,s_a,s)\\
    G_q^{(a,b)}(s_1,\ldots,s_a;s)&\coloneqq\begin{multlined}[t]\sum_{n=0}^\infty G_q^{(a+1,b-1)}(s_1,\ldots,s_a,s-n-b;n+b)\\-\sum_{n=0}^\infty G_q^{(a+1,b-1)}(s_1,\ldots,s_a,-n;s+n),\end{multlined}
  \end{align*}
  where $s_1,\ldots,s_a,s$ are complex variables such that $\Re(s)>b$, $\Re(s+s_{a-k+1}+\cdots+s_a)>k+b$ for $k=1,\ldots,a$.
\end{dfn}
\begin{rem}
  We note that $G_q^{(a,b)}(s_1,\ldots,s_a;s)$ is a sum of $q$MZFs of depth $a+b$.
\end{rem}
Using this notation, we have the following generalization for higher depths of Theorem \ref{thm1}.
\begin{thm}\label{thm2}
  Let $b$ be a positive integer. For $s\in\mathbb{C}$ with $\Re(s)>b$, we have 
  \[G_q^{(0,b)}(s)=\zeta_q(s).\]
\end{thm}
\begin{rem}
  If $s\in\mathbb{Z}_{>b}$, $G_q^{(a,b)}(s_1,\ldots,s_a;s)$ is equal to
  \[\sum_{\substack{m_1+\cdots+m_b=s\\m_1,\ldots,m_{b-1}\ge1\\m_b\ge2}}\zeta_q(s_1,\ldots,s_a,m_1,\ldots,m_b)\]
  by definition. Thus the case $s\in\mathbb{Z}_{>b}$ of the theorem implies Theorem \ref{Theorem 1}.
\end{rem}
\section*{Acknowledgement}
The author would like to express his sincere gratitude to Prof. Minoru Hirose for his many helpful suggestions and valuable advice.
The author is also grateful to the referee for a careful reading of the manuscript.
The author would like to thank Hanamichi Kawamura for his assistance in revising the manuscript. Thanks are also due to Takumi Maesaka, Nozomu Osanai, and Itto Matsubara for their valuable comments and advice.
This work was supported by the Academic Research Club of KADOKAWA DWANGO Educational Institute.
\section{Proof of Theorem \ref{thm1}}
In what follows, $s$ always stands for a complex number and $\sigma$ for its real part.
First, we prove Theorem \ref{thm1} for $\sigma>2$ by series transformation.
\begin{lem}\label{lem1}
Let $m_1$ be a positive integer. Then we have 
  \[
    \sum_{m>0}\left(\frac{1}{[m]_q}-\frac{1}{[m+m_1]_q}\right)=\left(\frac{1}{[1]_q}+\frac{1}{[2]_q}+\cdots+\frac{1}{[m_1]_q}\right)-m_1(1-q).
  \]
\end{lem}
\begin{proof}[Proof of Theorem \ref{thm1} for $\sigma>2$.]Since the functions $\zeta_q(s-n-2,n+2)$ and $\zeta_q(-n,s+n)$ are absolutely convergent in $\sigma>2$, the left-hand side of Theorem \ref{thm1} can be computed as 
\begin{align}
&\sum_{n=0}^\infty\sum_{0<m_1<m_2}\left(\frac{q^{(s-n-3)m_1+(n+1)m_2}}{[m_1]_q^{s-n-2}[m_2]_q^{n+2}}-\frac{q^{(-n-1)m_1+(s+n-1)m_2}}{[m_1]_q^{-n}[m_2]_q^{s+n}}\right)\notag\\&=\sum_{0<m_1<m_2}\sum_{n=0}^\infty\left(\frac{[m_1]_q}{[m_2]_q}q^{m_2-m_1}\right)^n\left(\frac{q^{(s-3)m_1+m_2}}{[m_1]_q^{s-2}[m_2]_q^{2}}-\frac{q^{-m_1+(s-1)m_2}}{[m_2]_q^{s}}\right)\notag\\
&=(1-q)^2\sum_{0<m_1<m_2}\frac{1}{(1-q^{m_2})^2-(1-q^{m_1})(1-q^{m_2})q^{m_2-m_1}}\left(\frac{q^{(s-3)m_1+m_2}}{[m_1]_q^{s-2}}-\frac{q^{-m_1+(s-1)m_2}}{[m_2]_q^{s-2}}\right).\label{label3}
\end{align}
Here, we denote the first and second terms of equation (\ref{label3}) by $(RHS_1)$ and $(RHS_2)$, respectively.  

$(RHS_1)$, $(RHS_2)$ can be computed as 
\begin{align}
(RHS_1)&=(1-q)^2\sum_{0<m_1<m_2}\frac{1}{(1-q^{m_2})^2-(1-q^{m_1})(1-q^{m_2})q^{m_2-m_1}}\frac{q^{(s-3)m_1+m_2}}{[m_1]_q^{s-2}}\notag\\&=(1-q)\sum_{0<m_1<m_2}\frac{q^{(s-2)m_1}}{[m_1]_q^{s-1}}\left(\frac{1}{(1-q^{m_2})-(1-q^{m_1})q^{m_2-m_1}}-\frac{1}{1-q^{m_2}}\right)\notag\\
&=\sum_{0<m_1}\frac{q^{(s-2)m_1}}{[m_1]_q^{s-1}}\sum_{m_1<m_2}\left(\frac{1-q}{1-q^{m_2-m_1}}-\frac{1-q}{1-q^{m_2}}\right)\notag\\
&=\sum_{0<m_1}\frac{q^{(s-2)m_1}}{[m_1]_q^{s-1}}\left(\frac{1}{[1]_q}+\frac{1}{[2]_q}+\cdots+\frac{1}{[m_1]_q}\right)-(1-q)\sum_{0<m_1}\frac{q^{(s-2)m_1}m_1}{[m_1]_q^{s-1}},\label{label1}
\end{align}

\begin{align}
(RHS_2)&=(1-q)^2\sum_{0<m_1<m_2}\frac{1}{(1-q^{m_2})^2-(1-q^{m_1})(1-q^{m_2})q^{m_2-m_1}}\frac{q^{-m_1+(s-1)m_2}}{[m_2]_q^{s-2}}\notag\\
&=(1-q)\sum_{1<m_2}\frac{q^{(s-1)m_2}}{[m_2]_q^{s-1}}\sum_{0<m_1<m_2}\frac{q^{-m_1}}{(1-q^{m_2})-(1-q^{m_1})q^{m_2-m_1}}\notag\\
&=(1-q)\sum_{1<m_2}\frac{q^{(s-2)m_2}}{[m_2]_q^{s-1}}\sum_{0<m_1<m_2}\frac{q^{m_2-m_1}}{1-q^{m_2-m_1}}\notag\\
&=\sum_{0<m_2}\frac{q^{(s-2)m_2}}{[m_2]_q^{s-1}}\left(\frac{q}{[1]_q}+\frac{q^2}{[2]_q}+\cdots+\frac{q^{m_2-1}}{[m_2-1]_q}\right).\label{label2}
\end{align}
In the computation of $(RHS_1)$, we used Lemma \ref{lem1}. Therefore, from equations (\ref{label1}) and (\ref{label2}), we can show
\begin{align}
&\sum_{n=0}^\infty\sum_{0<m_1<m_2}\left(\frac{q^{(s-n-3)m_1+(n+1)m_2}}{[m_1]_q^{s-n-2}[m_2]_q^{n+2}}-\frac{q^{(-n-1)m_1+(s+n-1)m_2}}{[m_1]_q^{-n}[m_2]_q^{s+n}}\right)\notag\\
&=\sum_{0<m_1}\frac{q^{(s-2)m_1}}{[m_1]_q^{s-1}}\left(\frac{1-q}{[1]_q}+\frac{1-q^2}{[2]_q}+\cdots+\frac{1-q^{m_1-1}}{[m_1-1]_q}+\frac{1}{[m_1]_q}\right)-(1-q)\sum_{0<m_1}\frac{q^{(s-2)m_1}m_1}{[m_1]_q^{s-1}}\notag\\
&=\sum_{0<m_1}\frac{q^{(s-2)m_1}}{[m_1]_q^{s-1}}\left(\frac{1}{[m_1]_q}-(1-q)\right)\notag\\
&=\zeta_q(s).
\end{align}
Finally we obtain the result.
\end{proof}
Next, we prove Theorem \ref{thm1} for $s$ satisfying $\sigma<2$ and $s\in\mathbb{C}\setminus\mathfrak{A}$ by using analytic continuation. The key estimate is the following lemma, whose proof follows the Mellin--Barnes method of Matsumoto (see, for example, \cite{Matsumoto1998, Matsumoto2002, Matsumoto2003}).
\begin{lem}\label{lem6}
Let $K\subset\mathbb{C}\setminus A$ be a compact set, let $T\ge0$, and let $N\ge2$ be an integer
with $\min_{s\in K}\Re(s)+N-\frac12>2$. Then there exists a constant $C_{q,K,N,T}>0$ such that
\[
  |\zeta_q(s-\alpha,\alpha)|\le C_{q,K,N,T}\,q^{\Re(\alpha)}(1+|\alpha|)^{N}
\]
for all $s\in K$ and all $\alpha\in\mathbb{C}$ with $\Re(\alpha)\ge2$ and $|\Im(\alpha)|\le T$.
\end{lem}
\begin{proof}
We use the Mellin--Barnes formula
\begin{align*}
(1+x)^{-\alpha}=\frac{1}{2\pi i}\int_{(c)}\frac{\Gamma(\alpha+z)\Gamma(-z)}{\Gamma(\alpha)}x^{z}dz
\end{align*}
where $\Re(\alpha)>0$, $|\arg x|<\pi$, $x\not=0$, $-\Re(\alpha)<c<0$ and $(c)$ denotes the vertical line $\Re(z)=c$ (see \cite[Section 14.51, p.289, Corollary]{ET}).
Let $f_q$ be defined by
\begin{align*}
f_q(s_1,\ldots,s_r;k_1,\ldots,k_r)&=\sum_{0<n_1<\cdots<n_r}\frac{q^{k_1n_1+\cdots+k_rn_r}}{[n_1]_q^{s_1}\cdots[n_r]_q^{s_r}}
\end{align*}
which is introduced in \cite{Zhao}.
First assume $\Re(\alpha)>1$, $\Re(s)>2$, and take $c$ with $\max\{-\Re(\alpha),2-\Re(s)\}<c<0$.
Since $[n_1+n_2]_q=[n_1]_q+q^{n_1}[n_2]_q$, we have
\begin{align*}
\zeta_q(s-\alpha,\alpha)
&=\sum_{0<n_1,n_2}\frac{q^{(s-\alpha-1)n_1+(\alpha-1)(n_1+n_2)}}{[n_1]_q^{s-\alpha}[n_1+n_2]_q^{\alpha}}
=\sum_{0<n_1,n_2}\frac{q^{(s-2)n_1+(\alpha-1)n_2}}{[n_1]_q^{s}}\Bigl(1+\frac{q^{n_1}[n_2]_q}{[n_1]_q}\Bigr)^{-\alpha}\\
&=\frac{1}{2\pi i}\int_{(c)}\frac{\Gamma(\alpha+z)\Gamma(-z)}{\Gamma(\alpha)}f_q(z+s;s-2+z)f_q(-z;\alpha-1)\,dz.
\end{align*}
The interchange of summation and integration is justified by absolute convergence:
on $\Re(z)=c$ the sum over $n_1$ converges since $\Re(s)+c>2$, the sum over $n_2$ converges since
$\Re(\alpha)>1$ and $[n_2]_q^{c}\le1$, and $\Gamma(\alpha+z)\Gamma(-z)$ decays exponentially in $|\Im(z)|$ by Stirling's formula.

Since $0<q<1$, we have $[n]_q=1+q+\cdots+q^{n-1}\ge1$. Hence, for $\Re(z)\le N-\varepsilon$ and $\Re(\alpha)\ge2$,
\begin{align*}
|f_q(-z;\alpha-1)|
\le\sum_{0<n}q^{(\Re(\alpha)-1)n}[n]_q^{N-\varepsilon}
\le q^{\Re(\alpha)-1}\sum_{0<n}q^{n-1}[n]_q^{N-\varepsilon}
=q^{\Re(\alpha)-2}f_q(-N+\varepsilon;1),
\end{align*}
where we used $q^{(\Re(\alpha)-1)n}=q^{\Re(\alpha)-1}q^{(\Re(\alpha)-1)(n-1)}\le q^{\Re(\alpha)-1}q^{n-1}$.
Similarly, if $\Re(s+z)\ge\sigma>2$, then
\begin{align*}
|f_q(z+s;s-2+z)|\le\sum_{0<n}q^{(\sigma-2)n}[n]_q^{-\sigma}=f_q(\sigma;\sigma-2).
\end{align*}
Both bounds are independent of $\Im(z)$.

In the strip $c\le\Re(z)\le N-\varepsilon$ we have $\Re(s)+c>2$, so $f_q(z+s;s-2+z)f_q(-z;\alpha-1)$ is holomorphic and bounded uniformly in $\Im(z)$ there, while $\Gamma(\alpha+z)\Gamma(-z)\to0$ as $|\Im(z)|\to\infty$ uniformly in the strip by Stirling's formula. Hence we may move the line of integration to $\Re(z)=N-\varepsilon$. Since $\Re(\alpha+z)>0$ in the strip, the only poles are those of $\Gamma(-z)$ at $z=0,1,\ldots,N-1$. We obtain
\begin{align}\label{eq:shifted}
\zeta_q(s-\alpha,\alpha)
=\sum_{k=0}^{N-1}\binom{-\alpha}{k}f_q(k+s;s-2+k)f_q(-k;\alpha-1)+I_{(N-\varepsilon)}(s,\alpha).
\end{align}
Where we defined $I_{(N-\varepsilon)}$ by
\begin{align*}
I_{(N-\varepsilon)}(s,\alpha)=\frac{1}{2\pi i}\int_{(N-\varepsilon)}\frac{\Gamma(\alpha+z)\Gamma(-z)}{\Gamma(\alpha)}f_q(z+s;s-2+z)f_q(-z;\alpha-1)\,dz.
\end{align*}
By \cite{Zhao}, $f_q(u;u-2)=\zeta_q(u)+(1-q)\zeta_q(u-1)$ holds and continues $f_q(u;u-2)$ meromorphically, so each
$f_q(k+s;s-2+k)$ is holomorphic for $s\notin\mathfrak{A}$ and by the bounds above $I_{(N-\varepsilon)}(s,\alpha)$ is holomorphic in $\Re(s)>2-N+\varepsilon$. Therefore \eqref{eq:shifted} holds for $\Re(\alpha)>1$, $\Re(s)>2-N+\varepsilon$, $s\notin\mathfrak{A}$ by analytic continuation.

Let $\Re(\alpha)\ge2$. Put $A_K\coloneqq\max_{0\le k<N}\sup_{s\in K}|f_q(k+s;s-2+k)|$, which is finite by continuity. Since $|\binom{-\alpha}{k}|\le(1+|\alpha|)^k$ and $|f_q(-k;\alpha-1)|\le(1-q)^{-k}\sum_{0<n}q^{(\Re(\alpha)-1)n}\le q^{\Re(\alpha)-1}(1-q)^{-k-1}$, we get
\begin{align*}
\Bigl|\sum_{k=0}^{N-1}\binom{-\alpha}{k}f_q(k+s;s-2+k)f_q(-k;\alpha-1)\Bigr|
\le\frac{N A_K}{(1-q)^{N}}\,q^{\Re(\alpha)-1}(1+|\alpha|)^{N-1}.
\end{align*}
For $I^{(N-\varepsilon)}$, note that $|\Gamma(x+iy)|\le\Gamma(x)$ for $x>0$ and $y\in\mathbb{R}$, so that $|\Gamma(\alpha+z)|\le\Gamma(\Re(\alpha)+N-\varepsilon)$ on the line $\Re(z)=N-\varepsilon$. On the other hand, by Stirling's formula we have $\Gamma(\Re(\alpha)+N-\varepsilon)\ll_N\Gamma(\Re(\alpha))(1+|\alpha|)^{N-\varepsilon}$ and $|\Gamma(\alpha)|\gg_T\Gamma(\Re(\alpha))$ for $\Re(\alpha)\ge2$ and $|\Im(\alpha)|\le T$.
Hence 
\[\Bigl|\frac{\Gamma(\alpha+z)}{\Gamma(\alpha)}\Bigr|\ll_{N,T}(1+|\alpha|)^{N-\varepsilon} \qquad(\Re(z)=N-\varepsilon)\]
holds.

Since $N-\varepsilon\notin\mathbb{Z}$, $\Gamma(-z)$ is continuous on $\Re(z)=N-\varepsilon$, and $|\Gamma(-N+\varepsilon-it)|\ll|t|^{-N+\varepsilon-1/2}e^{-\pi|t|/2}$ as $|t|\to\infty$ by Stirling's formula, we have
\begin{align*}
J_{(N-\varepsilon)}\coloneqq\int_{(N-\varepsilon)}|\Gamma(-z)||dz|<\infty .
\end{align*}
Therefore
\begin{align*}
|I_{(N-\varepsilon)}(s,\alpha)|
\ll J_{(N-\varepsilon)}\,f_q(\sigma_K+N-\varepsilon;\sigma_K+N-\varepsilon-2)\,f_q(-N+\varepsilon;1)\,
q^{\Re(\alpha)-2}(1+|\alpha|)^{N-\varepsilon}.
\end{align*}
One may choose $\varepsilon=1/2$. Combining these, we complete the proof.
\end{proof}

\begin{proof}[Proof of Theorem \ref{thm1}.]
Let $K\subset\mathbb{C}\setminus A$ be compact and put $T=\max_{s\in K}|\Im(s)|$. Applying Lemma \ref{lem6} with $\alpha=n+2$, and with $\alpha=s+n$ for $n\ge2-\min_{s\in K}\Re(s)$, we see that $\sum_{n=0}^{\infty}\zeta_q(s-n-2,n+2)$ and $\sum_{n=0}^{\infty}\zeta_q(-n,s+n)$ converge uniformly on $K$. Since each term is holomorphic on $\mathbb{C}\setminus A$ by \cite{Zhao}, the left-hand side of Theorem \ref{thm1} is holomorphic on the connected set $\mathbb{C}\setminus A$, and coincides with $\zeta_q(s)$ for $\Re(s)>2$. This completes the proof.
\end{proof}
\section{Proof of Theorem \ref{thm2}}
\begin{dfn}For a positive integer $d$, a non-negative integer $D$, and $s\in\mathbb{C}$, we define
  \begin{align*}
    F_q^{(0)}(D;s;d)&\coloneqq\sum_{D<m}\frac{q^{(s-d-1)m}}{[m]_q^{s-d}}\sum_{m-D\le n_1\le \cdots\le n_d\le m}\frac{q^{n_1+\cdots+n_d}}{[n_1]_q\cdots[n_d]_q},\\
    F_q^{(1)}(D;s;d)&\coloneqq\sum_{D<t<m}\frac{q^{(s-d-2)t}\cdot q^{m-t}}{[t]_q^{s-d-1}[m-t]_q}\sum_{m-t\le n_1\le\cdots\le n_d\le m}\frac{q^{n_1+\cdots+n_d}}{[n_1]_q\cdots[n_d]_q},\\
    F_q^{(2)}(D;s;d)&\coloneqq\sum_{D<t<m}\frac{q^{(s-d-2)t}\cdot q^m}{[t]_q^{s-d-1}[m]_q}\sum_{m-t\le n_1\le\cdots\le n_d\le m}\frac{q^{n_1+\cdots+n_d}}{[n_1]_q\cdots[n_d]_q},\\
    F_q^{(3)}(D;s;d)&\coloneqq\sum_{D<t<m}\frac{q^{(s-d-2)m}\cdot q^{m-t}}{[m]_q^{s-d-1}[m-t]_q}\sum_{m-t\le n_1\le\cdots\le n_d\le m}\frac{q^{n_1+\cdots+n_d}}{[n_1]_q\cdots [n_d]_q}.
  \end{align*}
\end{dfn}
\begin{lem}\label{Flem1}
  The function $F_q^{(0)}(D;s;d)$ converges absolutely when $\sigma>1$. In addition, if $\sigma>d+2$, the functions $F_q^{(1)}(D;s;d)$, $F_q^{(2)}(D;s;d)$, and $F_q^{(3)}(D;s;d)$ converge absolutely. Moreover, we have
  \begin{equation}
    \left|F_q^{(i)}(D;s;d)\right|\ll\sum_{D<t}\frac{q^{(\sigma-d-2)t}(\log t)^{d+1}}{[t]_q^{\sigma-d-1}}\label{label4}
  \end{equation}
  for $i=1,2,3$ and the implicit constant is independent of both $D$ and $s$.
\end{lem}
\begin{proof}
  The convergence of $F_q^{(0)}(D;s;d)$ is immediate. The convergence of $F_q^{(i)}(D;s;d)$ follows from equation (\ref{label4}).
  Since
  \[\left|F_q^{(i)}(D;s;d)\right|\le F_q^{(1)}(D;\sigma;d),\]
  it is enough to prove equation (\ref{label4}) for $i=1$ and $s\in\mathbb{R}_{>d+2}$. 
Write $F_q^{(1)}(D;s;d)$ as $A+B$ where

\begin{align*}
  A&\coloneqq\sum_{D<t\le \frac{m}{2}}\frac{q^{(s-d-2)t}\cdot q^{m-t}}{[t]_q^{s-d-1}[m-t]_q}\sum_{m-t\le x_1\le\cdots\le x_d\le m}\frac{q^{x_1+\cdots+x_d}}{[x_1]_q\cdots[x_d]_q},\\
  B&\coloneqq\sum_{\substack{D<t\\\frac{m}{2}<t<m}}\frac{q^{(s-d-2)t}\cdot q^{m-t}}{[t]_q^{s-d-1}[m-t]_q}\sum_{m-t\le x_1\le\cdots\le x_d\le m}\frac{q^{x_1+\cdots+x_d}}{[x_1]_q\cdots[x_d]_q}.
\end{align*}  
Since $q^{m-1}m\le [m]_q$ hold, we have 
\[\sum_{m-t\le x_1\le\cdots\le x_d\le m}\frac{q^{x_1+\cdots+x_d}}{[x_1]_q\cdots[x_d]_q}\le\sum_{m-t\le x_1\le\cdots\le x_d\le m}\frac{1}{x_1\cdots x_d}\ll\frac{t^d}{(m-t)^d},\]
and
\begin{align*}
  A&\ll\sum_{D<t\le\frac{m}{2}}\frac{q^{(s-d-2)t}}{[t]_q^{s-d-1}t^{-d}(m-t)^{d+1}}\\
  &\ll\sum_{D<t\le\frac{m}{2}}\frac{q^{(s-d-2)t}}{[t]_q^{s-d-1}t^{-d}m^{d+1}}\\
  &\ll\sum_{D<t}\frac{q^{(s-d-2)t}}{[t]_q^{s-d-1}}.
\end{align*}
 We also have
\[ B\ll\sum_{\substack{D<t\\\frac{m}{2}<t<m}}\frac{(\log m)^d}{[t]_q^{s-d-1}[m-t]_q}\ll\sum_{D<t<m<2t}\frac{q^{(s-d-2)t}\cdot q^{m-t}(\log m)^d}{[t]_q^{s-d-1}[m-t]_q}.\]
Thus putting $n=m-t$, we get
\[B\ll\sum_{\substack{D<t\\1\le n<t}}\frac{q^{(s-d-2)t}\cdot q^{n} (\log(2t))^d}{[t]_q^{s-d-1}[n]_q}\ll\sum_{D<t}\frac{q^{(s-d-2)t}\cdot (\log(2t))^d\log t}{[t]_q^{s-d-1}}.\]
This finishes the proof.

\end{proof}
\begin{lem}\label{Flem2}If $\sigma>d+2$, we have
  \begin{equation}
    F_q^{(0)}(D;s;d+1)=F_q^{(1)}(D;s;d)-F_q^{(2)}(D;s;d)-F_q^{(3)}(D;s;d).
  \end{equation}
\end{lem}
\begin{proof}
Applying Lemma \ref{Flem1}, $F_q^{(i)}(D;s;d)$ converges absolutely for $i=1,2,3$.
We can show
\begin{align*}
   F_q^{(1)}(D;s;d)&=\sum_{D<t<m}\frac{q^{(s-d-2)t}\cdot q^{m-t}}{[t]_q^{s-d-1}[m-t]_q}\sum_{m-t\le n_1\le\cdots\le n_d\le m}\frac{q^{n_1+\cdots+n_d}}{[n_1]_q\cdots[n_d]_q}\\
    &=\sum_{D<t}\frac{q^{(s-d-2)t}}{[t]_q^{s-d-1}}\sum_{0<n_0\le n_1\le\cdots\le n_d\le n_0+t}\frac{q^{n_0+\cdots+n_d}}{[n_0]_q[n_1]_q\cdots[n_d]_q},\\
F_q^{(2)}(D;s;d)&=\sum_{D<t}\frac{q^{(s-d-2)t}}{[t]_q^{s-d-1}}\sum_{0<n_{d+1}-t \le n_1\le\cdots\le n_d\le n_{d+1}}\frac{q^{n_1+\cdots+n_d+n_{d+1}}}{[n_1]_q\cdots[n_d]_q[n_{d+1}]_q}\\
&=\sum_{D<t}\frac{q^{(s-d-2)t}}{[t]_q^{s-d-1}}\sum_{\substack{0<n_0\le n_1\le \cdots\le n_d\\n_d>t\\n_d\le n_0+t}}\frac{q^{n_0+\cdots+n_d}}{[n_0]_q[n_1]_q\cdots[n_d]_q},\\
    F_q^{(3)}(D;s;d)&=\sum_{D<t<N}\frac{q^{(s-d-2)N}\cdot q^{N-t}}{[N]_q^{s-d-1}[N-t]_q}\sum_{N-t\le n_1\le\cdots\le n_d\le N}\frac{q^{n_1+\cdots+n_d}}{[n_1]_q\cdots [n_d]_q}\\
    &=\sum_{D<N}\frac{q^{(s-d-2)N}}{[N]_q^{s-d-1}}\sum_{\substack{0<n_0\le n_1\le\cdots\le n_d\le N\\n_0<N-D}}\frac{q^{n_0+\cdots+n_d}}{[n_0]_q[n_1]_q\cdots [n_d]_q}\\
    &=\sum_{D<t}\frac{q^{(s-d-2)t}}{[t]_q^{s-d-1}}\sum_{\substack{0<n_0\le n_1\le\cdots\le n_d\le t\\n_0<t-D}}\frac{q^{n_0+\cdots+n_d}}{[n_0]_q[n_1]_q\cdots [n_d]_q},
  \end{align*}
  and we have
  \begin{align*}
    &F_q^{(1)}(D;s;d)-F_q^{(2)}(D;s;d)-F_q^{(3)}(D;s;d)\\&=\sum_{D<t}\frac{q^{(s-d-2)t}}{[t]_q^{s-d-1}}\sum_{t-D\le n_0\le n_1\le \cdots\le n_d\le t}\frac{q^{n_0+\cdots+n_d}}{[n_0]_q[n_1]_q\cdots[n_d]_q}\\
    &=F_q^{(0)}(D;s;d+1).
  \end{align*}
  Thus we obtain the result.
\end{proof}
\begin{lem}\label{Blem}
Here $\sigma_i$ denote the real part of $s_i$.
Let $a$ be a non-negative integer and $b$ an integer with $b\ge2$. Assume that $\sigma>b$ and \[\sigma+\sigma_{a-k+1}+\cdots \sigma_a>k+b\] hold for $k=1,\ldots,a$, then $G_q^{(a,b)}(s_1,\ldots,s_a;s)$ is a well-defined function and  
\begin{equation}
  G_q^{(a,b)}(s_1,\ldots,s_a;s)=\sum_{0<m_1<\cdots<m_a<m}\frac{q^{(s_1-1)m_1+\cdots+(s_a-1)m_a+(s-b)m}}{[m_1]_q^{s_1}\cdots[m_a]_q^{s_a}[m]_q^{s-b+1}}\sum_{m-m_a\le x_1\le\cdots\le x_{b-1}\le m}\frac{q^{x_1+\cdots+x_{b-1}}}{[x_1]_q\cdots[x_{b-1}]_q}\label{label8}
\end{equation}
hold. Here we understand that $m-m_a=m$ if $a=0$.
\end{lem}
\begin{proof}
Note that the right-hand side of equation (\ref{label8}) converges absolutely. The proof is by induction on $b$. We can prove the case $b=2$ in the similar manner as in the proof of Theorem \ref{thm1} for $\sigma>2$. For $b\ge 3$, by the induction hypothesis, we have
  \begin{align*}
    &G_q^{(a,b)}(s_1,\ldots,s_a;s)\\&=\sum_{n=0}^\infty G_q^{(a+1,b-1)}(s_1,\ldots,s_a,s-n-b;n+b)-\sum_{n=0}^\infty G_q^{(a+1,b-1)}(s_1,\ldots,s_a,-n;s+n)\\
    &=\begin{multlined}[t]\sum_{n=0}^\infty\sum_{0<m_1<\cdots<m_{a+1}<m}\frac{q^{(s_1-1)m_1+\cdots+(s_a-1)m_a+(s-n-b-1)m_{a+1}+(n+1)m}}{[m_1]_q^{s_1}\cdots[m_a]_q^{s_a}[m_{a+1}]_q^{s-n-b}[m]_q^{n+2}}\\\times\sum_{m-m_{a+1}\le x_1\le\cdots\le x_{b-2}\le m}\frac{q^{x_1+\cdots+x_{b-2}}}{[x_1]_q\cdots[x_{b-2}]_q}\\-\sum_{n=0}^\infty\sum_{0<m_1<\cdots<m_{a+1}<m}\frac{q^{(s_1-1)m_1+\cdots+(s_a-1)m_a+(-n-1)m_{a+1}+(s+n-b+1)m}}{[m_1]_q^{s_1}\cdots[m_a]_q^{s_a}[m_{a+1}]_q^{-n}[m]_q^{s+n-b+2}}\\\times\sum_{m-m_{a+1}\le x_1\le\cdots\le x_{b-2}\le m}\frac{q^{x_1+\cdots+x_{b-2}}}{[x_1]_q\cdots[x_{b-2}]_q}.\end{multlined}
  \end{align*}
  Then we have
  \begin{align*}
&\begin{multlined}[t]\sum_{n=0}^\infty\sum_{0<m_1<\cdots<m_{a+1}<m}\frac{q^{(s_1-1)m_1+\cdots+(s_a-1)m_a+(s-n-b-1)m_{a+1}+(n+1)m}}{[m_1]_q^{s_1}\cdots[m_a]_q^{s_a}[m_{a+1}]_q^{s-n-b}[m]_q^{n+2}}\\\times\sum_{m-m_{a+1}\le x_1\le\cdots\le x_{b-2}\le m}\frac{q^{x_1+\cdots+x_{b-2}}}{[x_1]_q\cdots[x_{b-2}]_q}\end{multlined}\\
    &=\begin{multlined}[t]\sum_{0<m_1<\cdots<m_{a+1}<m}\frac{q^{(s_1-1)m_1+\cdots+(s_a-1)m_a+(s-b)m_{a+1}}}{[m_1]_q^{s_1}\cdots[m_a]_q^{s_a}[m_{a+1}]_q^{s-b+1}}\left(\frac{q^{m-m_{a+1}}}{[m-m_{a+1}]_q}-\frac{q^m}{[m]_q}\right)\\\times\sum_{m-m_{a+1}\le x_1\le\cdots\le x_{b-2}\le m}\frac{q^{x_1+\cdots+x_{b-2}}}{[x_1]_q\cdots[x_{b-2}]_q}\end{multlined}\\
    &=\sum_{0<m_1<\cdots<m_a}\frac{q^{(s_1-1)m_1+\cdots+(s_a-1)m_a}}{[m_1]_q^{s_1}\cdots[m_a]_q^{s_a}}\left(F_q^{(1)}(m_a;s;b-2)-F_q^{(2)}(m_a;s;b-2)\right),
  \end{align*}
  and
  \begin{align*}
    &\begin{multlined}[t]\sum_{n=0}^\infty\sum_{0<m_1<\cdots<m_{a+1}<m}\frac{q^{(s_1-1)m_1+\cdots+(s_a-1)m_a+(-n-1)m_{a+1}+(s+n-b+1)m}}{[m_1]_q^{s_1}\cdots[m_a]_q^{s_a}[m_{a+1}]_q^{-n}[m]_q^{s+n-b+2}}\\\times\sum_{m-m_{a+1}\le x_1\le\cdots\le x_{b-2}\le m}\frac{q^{x_1+\cdots+x_{b-2}}}{[x_1]_q\cdots[x_{b-2}]_q}\end{multlined}\\
    &=
    \begin{multlined}[t]\sum_{0<m_1<\cdots<m_{a+1}<m}\frac{q^{(s_1-1)m_1+\cdots+(s_a-1)m_a+(s-b)m}\cdot q^{m-m_{a+1}}}{[m_1]_q^{s_1}\cdots[m_a]_q^{s_a}[m]_q^{s-b+1}[m-m_{a+1}]_q}\\\times\sum_{m-m_{a+1}\le x_1\le\cdots\le x_{b-2}\le m}\frac{q^{x_1+\cdots+x_{b-2}}}{[x_1]_q\cdots[x_{b-2}]_q}\end{multlined}\\
    &=\sum_{0<m_1<\cdots<m_a}\frac{q^{(s_1-1)m_1+\cdots+(s_a-1)m_a}}{[m_1]_q^{s_1}\cdots[m_a]_q^{s_a}}F_q^{(3)}(m_a;s;b-2).
  \end{align*}
   Since the series 
   \[\sum_{0<m_1<\cdots<m_a}\frac{q^{(s_1-1)m_1+\cdots+(s_a-1)m_a}}{[m_1]_q^{s_1}\cdots [m_a]_q^{s_a}}F_q^{(i)}(m_a;s;b-2)\]
   for $i=1,2,3$ is absolutely convergent by Lemma \ref{Flem1}, we see that $G_q^{(a,b)}(s_1,\ldots,s_a;s)$ is also absolutely convergent, and furthermore, we have
  \begin{align*}
    &G_q^{(a,b)}(s_1,\ldots,s_a;s)\\
    &=\begin{multlined}[t]\sum_{0<m_1<\cdots<m_a}\frac{q^{(s_1-1)m_1+\cdots+(s_a-1)m_a}}{[m_1]_q^{s_1}\cdots[m_a]_q^{s_a}}\\\times\left(F_q^{(1)}(m_a;s;b-2)-F_q^{(2)}(m_a;s;b-2)-F_q^{(3)}(m_a;s;b-2)\right)\end{multlined}\\
    &=\sum_{0<m_1<\cdots<m_a}\frac{q^{(s_1-1)m_1+\cdots+(s_a-1)m_a}}{[m_1]_q^{s_1}\cdots[m_a]_q^{s_a}}F_q^{(0)}(m_a;s;b-1)\\
    &=\sum_{0<m_1<\cdots<m_a<m}\frac{q^{(s_1-1)m_1+\cdots+(s_a-1)m_a+(s-b)m}}{[m_1]_q^{s_1}\cdots[m_a]_q^{s_a}[m]_q^{s-b+1}}\sum_{m-m_a\le x_1\le\cdots\le x_{b-1}\le m}\frac{q^{x_1+\cdots+x_{b-1}}}{[x_1]_q\cdots[x_{b-1}]_q}
  \end{align*}
  from Lemma \ref{Flem2}. Hence the lemma is proved.
\end{proof}
\begin{proof}[Proof of Theorem \ref{thm2}]
From Lemma \ref{Blem}, we have
  \begin{align*}
    G_q^{(0,b)}(s)&=\sum_{0<m}\frac{q^{(s-b)m}}{[m]_q^{s-b+1}}\sum_{m\le x_1\le\cdots\le x_{b-1}\le m}\frac{q^{x_1+\cdots+x_{b-1}}}{[x_1]_q\cdots[x_{b-1}]_q}\\
    &=\zeta_q(s).
  \end{align*}
  This completes the proof.
\end{proof}


\begin{thebibliography}{99}
\bibitem{Dav}D.~M.~Bradley, On the sum formula for multiple $q$-zeta values, Rocky Mountain J.~Math.~37 (2007), 1427--1434.
\bibitem{HMO}M.~Hirose, H.~Murahara, T.~Onozuka, Sum formula for multiple zeta function, Ramanujan J.~65 (2024), 1607--1619.
\bibitem{Matsumoto1998}K. Matsumoto, Asymptotic series for double zeta, double gamma, and Hecke $L$-functions,
Math. Proc. Cambridge Philos. Soc. \textbf{123} (1998), 385--405;
Corrigendum and addendum, ibid. \textbf{132} (2002), 377--384.
\bibitem{Matsumoto2002}K. Matsumoto, On the analytic continuation of various multiple zeta-functions,
in: Number Theory for the Millennium II (Urbana, IL, 2000), M. A. Bennett et al. (eds.),
A K Peters, Natick, MA, 2002, 417--440.
\bibitem{Matsumoto2003}K. Matsumoto, Asymptotic expansions of double zeta-functions of Barnes, of Shintani,
and Eisenstein series, Nagoya Math. J. \textbf{172} (2003), 59--102.
\bibitem{ET} E.~T.~Whittaker and G.~N.~Watson. A Course of Modern Analysis. Cambridge University Press, Cambridge, (1927).
\bibitem{Zhao}J.~Zhao, $q$-multiple zeta functions and $q$-multiple polylogarithms, Ramanujan J.~(2007) 14, 189--221.
\end{thebibliography}
\end{document}